\documentclass{amsart}

\usepackage{amsmath,amssymb,amsfonts,graphics,color}
\usepackage{epsfig}
\usepackage{latexsym}
\usepackage{euscript}

\title{Linear embeddings of graphs and graph limits}

\author[Chuangpishit et al.]{Huda Chuangpishit \and Mahya Ghandehari\and Matt Hurshman \and Jeannette Janssen \and Nauzer Kalyaniwalla}
\address{Department of Mathematics \& Statistics, Dalhousie University, Halifax, Nova Scotia, Canada, B3H 3J5}


\date{\today}

\newtheorem{theorem}{Theorem}[section]
\newtheorem{corollary}[theorem]{Corollary}
\newtheorem{definition}[theorem]{Definition}

\newtheorem{conjecture}[theorem]{Conjecture}

\newtheorem{lemma}[theorem]{Lemma}
\newtheorem{proposition}[theorem]{Proposition}

\newcommand{\cal}{\mathcal}
\newcommand{\Rrr}{{\mathbb R}}

\def\sup{{\rm sup}}
\def\Pr{{\rm Pr}}

\def\Nnn{{\mathbb N}}

\begin{document}

\begin{abstract} Consider a random graph process where vertices are chosen from the interval $[0,1]$, and edges are chosen independently at  random, but so that, for a given vertex $x$, the probability that there is an edge to a vertex $y$ decreases as the distance between $x$ and $y$ increases. We call this a random graph with a linear embedding. 

We define a new graph parameter $\Gamma^*$, which aims to measure the similarity of the graph to an instance of a random graph with a linear embedding. For a graph $G$, $\Gamma^*(G)=0$ if and only if $G$ is a unit interval graph, and thus a deterministic example of a graph with a linear embedding. 

We show that the behaviour of $\Gamma^*$ is consistent with the notion of convergence as defined in the theory of dense graph limits. In this theory, graph sequences converge to a symmetric, measurable function on $[0,1]^2$. 
We define an operator $\Gamma$ which applies to graph limits, and which assumes the value zero precisely for graph limits that have a linear embedding. We show that, if a graph sequence $\{ G_n\}$ converges to a function $w$, then $\{ \Gamma^*(G_n)\}$ converges as well. Moreover, there exists a function $w^*$ arbitrarily close to $w$ under the box distance, so that $\lim_{n\rightarrow \infty}\Gamma^*(G_n)$ is arbitrarily close to $\Gamma (w^*)$.
\end{abstract}


\maketitle

\section{Introduction}
Consider the following random graph model on $n$ vertices. Vertices are randomly chosen from the interval $[0,1]$ according to a given distribution. Then, for each pair of vertices $x,y$, independently, an edge is added with probability $w(x,y)$, where $w:[0,1]^2\rightarrow [0,1]$ is a symmetric, measurable function.

In this article, we are interested in the special case where $w$ is increasing towards the diagonal. Specifically, for $x<y$, $w(x,y)$ decreases as $y$ increases or $x$ decreases. Such a random graph has a linear geometric interpretation: vertices are embedded in the line segment $[0,1]$, and live in a probability landscape where link probabilities decrease as the linear distance between points increases. We will refer to this as a random graph with a {\sl linear embedding}. 

Consider now the problem of recognizing graphs produced by a random graph process with a linear embedding. If the labels of the vertices are provided, this question may be answered by regular statistical methods. When only the isomorphism type of the graph is given, the question becomes 
more complicated. We address the question of how to recognize graphs whose structure is consistent with that of a random graph with a linear embedding. 

Recognition is easy in the special case of {\sl unit interval graphs}, or {\sl one-dimensional geometric graphs}. Here, the selection of vertices is random, but the edge formation is deterministic. In other words, the function $w$ governing edge formation only takes values in $\{ 0,1\}$. In this paper, we introduce a graph parameter $\Gamma^*$ which aims to measure  the similarity of the graph to an instance of a random graph with a linear embedding.
We show that $\Gamma^*$ of a given graph equals zero 
if and only if the graph is a one-dimensional geometric graph (Proposition \ref{Gamma*-geometric}). We then consider the behaviour of  $\Gamma^*$ when it is applied to convergent sequences of graphs $\{ G_n\}$, where convergence is defined as in the theory of graph limits as developed by Lov{\'a}sz and Szegedy in \cite{lovaszszegedy2006}.  

In this theory, convergence is defined based on homomorphism densities, and the limit is a symmetric, measurable function. The theory is developed and extended to sequences of random graphs by Borgs {\em et al.}~in \cite{borgsI2008,borgsII2007,borgs2011} and is  explored further by Lov{\'a}sz and  others (see for example \cite{borgs2010,chung2012,lovaszszegedy2010}. See also the recent book \cite{lovaszBook2012}). 
As shown by Diaconis and Janson in \cite{diaconisjanson2008}, the theory of graph limits is closely connected to the probabilistic theory of exchangeable arrays. A different view, where the limit object is referred to as a {\sl kernel}, is provided by Bollob{\'a}s, Janson and Riordan in \cite{bollobas2007,bollobasjansonriordan2009}. The connection with the results of Borgs {\sl et al.}~and an extension of the theory to sparse graphs are presented in \cite{bollobas2009}.

Homomorphism densities characterize the isomorphism type of a (twine-free) graph. A graph sequence $\{ G_n\}$ converges if and only if all of the homomorphism densities of the graphs $G_n$ converge. Moreover, the limits of all these homomorphism densities can be obtained from a symmetric, measurable function $w$ on $[0,1]^2$ which represents the ``limit object''. Thus, $w$ encapsulates the local structure of the graphs in the sequence. 
Conversely, the randomly growing graph sequence  obtained from $w$, according to the process described earlier, will asymptotically exhibit the same homomorphism densities, and thus have a similar structure. 

Let $\{G_n\}$ be a sequence of graphs converging to a symmetric, measurable function $w$. (One may think of this sequence as an instance of a randomly growing graph sequence generated by $w$.) Can we recognize whether this sequence is generated by a random graph process with a linear embedding?
To answer this question, we introduce a parameter $\Gamma$, which applies to symmetric, measurable functions. For such a function $w$, $\Gamma(w)=0$ if and only if the function $w$ is diagonally increasing (Proposition \ref{prop:Gamma0iffDI}). A random graph process with a linear embedding is simply one for which the corresponding function $w$ satisfies $\Gamma(w)=0$.

The main result in this paper regards the relation between $\Gamma^*$ as applied to a convergent graph sequence $\{ G_n\}$, and $\Gamma$ applied to the limiting function $w$. Firstly, every graph $G$ can also be regarded as a $\{0,1\}$-valued function $w_G$. It is not hard to prove that, for a given graph $G$,  $\Gamma^*(G)$ and $\Gamma (w_G)$ are asymptotically equal (Theorem \ref{thm:approximation-of-gamma} and Corollary \ref{cor:gamma-graph-function}). A harder question concerns the relation between the sequence of $\Gamma^*$-values of the graphs, $\{\Gamma^*(G_n)\}$, and the $\Gamma$-value of the limiting function, $\Gamma (w)$. This question is addressed in Section \ref{continuity}.
To obtain any continuity type results, we need to address the fact that functions $w$ representing the limit of a converging graph sequence are not unique. Moreover, $\Gamma$ can attain different values for different functions representing the same limit object. Thus, we introduce $\tilde{\Gamma}$ as the infimum of $\Gamma(w)$, where the infimum is  taken over  equivalence classes of functions that all have box distance 0 to each other. Note that every equivalence class consists of functions that all represent the same limit object.

Our main result (Theorem \ref{thm:Gamma-cts}) shows that $\tilde{\Gamma}$ is continuous. It follows that, for a  graph sequence $\{G_n\}$ converging to a function $w$, the sequence $\{ \Gamma^* (G_n)\}$ converges to $\tilde{\Gamma}(w)$, the infimum of $\Gamma(w^*)$ over all functions $w^*$  which represent the limit of the converging sequence $\{G_n\}$. Thus,  there exists a function $w^*$ arbitrarily close to $w$ under the box distance, so that $\lim_{n\rightarrow\infty}\Gamma^*(G_n)$ is arbitrarily close to $\Gamma (w^*)$.

Our findings  justify the conclusion that, for large graphs,  $\Gamma^*(G)$ does give an indication of compatibility of $G$ with a random graph model with linear embedding. In particular for a converging graph sequence $\{G_n\}$, we have  $\Gamma^*(G_n)  \rightarrow 0$ as $n \rightarrow \infty$ {\sl if and only if} $\{ G_n\}$ converges to a function $w$ which has $\Gamma $-value arbitrarily small (Corollary \ref{cor:contin_Gamma*}).

The approach we take in this paper  
was inspired by a paper by Bollob{\' a}s, Janson and Riordan on monotone graph limits (see \cite{bollobas2011}). In that paper, a graph parameter $\Omega$ is introduced, which assumes value zero precisely for threshold graphs. It is then shown that a converging sequence of graphs for which $\Omega$ tends to zero has a limit that is a {\em monotone} function. Thus, monotone graph limits can be seen as generalizations of threshold graphs. 

The flavour of the results in this paper is similar to those on monotone graph limits. Namely, we show  that diagonally increasing graph limits can be seen as generalizations of unit interval graphs.  
However, monotone functions have ``nice'' properties that do not carry over to diagonally increasing functions.
So  there are significant differences where the proofs are concerned. Specifically,   the equivalence class of functions obtained by applying measure preserving maps to a given function $w$ contains at most one monotone function. This is not true for diagonally increasing functions, which is why we need to introduce the parameter $\tilde{\Gamma}$, which complicates the statement and proof of the main result. 
Another major difference is that, for monotone functions, $L^1$-distance and box distance are equivalent. This however is not true for diagonally increasing functions. Thus we need to use entirely different methods to prove our continuity result than the ones developed in \cite{bollobas2011}.

Diaconis, Holmes and Janson also consider the limits of threshold graphs (see \cite{diaconis2008}), and  the limits of interval graphs (see \cite{diaconis2011}). Note that the one-dimensional geometric graphs studied in our paper are a special class of interval graphs; namely unit interval or proper interval graphs. However, the authors of \cite{diaconis2011} focus on different properties and generalizations of interval graphs, and their results do not apply to the problems we consider here. 

Finally, we say a few words about the motivation behind this paper. Our results show that a graph parameter, $\Gamma^*$, applied to graphs of increasing size, can help recognize graphs that are ``close'' to a diagonally increasing function, and thus resemble a random graph with a linear embedding. Therefore, we can interpret $\Gamma^*$ as a parameter that helps recognize the (one-dimensional) spatial embedding underlying the graph. 

The question of recognizing graphs that have a spatial embedding is motivated by the study of real-life complex networks. 
If one assumes that such networks are the manifestation 
of an underlying reality, then a useful way to model these networks is to take a latent space approach.  In this approach, the formation of the graph is informed by the hidden spatial reality. The graph formation is modelled as a stochastic process, where the probability of a link occurring  between two vertices decreases as their metric distance increases. 

The spatial reality can be used to represent attributes of the vertices which are inaccessible or unknown, but which are assumed to inform link formation. For example, in a social network, vertices may be considered as members of a social space, where the coordinates represent the interests and background of the users. Given only the graph, such a spatial model allows us to mine the underlying spatial reality. This approach was taken by Hoff {\sl et al.}~in \cite{hoff2002}.  In most cases, spatial models are formed on spaces of dimension at least two, but a one-dimensional (linear) spatial  model, the {\em niche model}, is proposed in \cite{williamsmartinez2000} to model food webs. 
Our result can be interpreted as a step towards the recognition of graphs that can be well-modelled by a linear spatial model.

This paper is organized as follows.  In Section \ref{Background}, we briefly review the results from the theory of graph limits. 
In Section \ref{Gamma*}, we give precise definitions for the concepts of spatial embedding and linear embedding for a random graph model, introduce the graph parameter $\Gamma^*$, and show that it characterizes  one-dimensional geometric graphs. In Section \ref{Gamma}, we introduce a continuous analogue of $\Gamma^*$, called $\Gamma $, which applies to symmetric measurable functions.
In Section \ref{relation} we show that, for any graph $G$, $\Gamma^* (G)$ is asymptotically equal to the value of $\Gamma$ applied to the $\{ 0,1\}$-valued function representing $G$. In Section \ref{continuity}  we introduce the generalized parameter $\tilde{\Gamma}$. Our main result is Theorem  \ref{thm:Gamma-cts} which shows that $\tilde{\Gamma}$ is continuous. In Corollary \ref{cor:contin_Gamma*}, we interpret this continuity result for converging graph sequences.

\section{Preliminaries: graph limits}

\label{Background}
In this section we summarize the basic definitions and results from the theory of graph limits, insofar as they are relevant to this paper. For more background, the reader is referred to 
the papers referenced in the introduction. A thorough study of the subject can be found in \cite{lovaszBook2012}.
In this section, we follow the terminology of \cite{lovaszszegedy2006}.

Let $F$ and $G$ be two simple graphs, {\it i.e.} graphs without loops or multiple edges. Let $V(F)$ and $V(G)$ be  vertex sets of $F$ and $G$ respectively. A map $V(F)\rightarrow V(G)$ is called a {\em homomorphism} from $F$ to $G$ if it maps adjacent vertices in $F$ to adjacent vertices 
in $G$. Let $hom (F, G)$ be the number of homomorphisms of $F$ into $G$. The {\em homomorphism density} of $F$ into $G$ is defined as
 \[ 
 t(F,G)=\frac{hom(F,G)}{|V(G)|^{|V(F)|}}.
 \]
The homomorphism density can be interpreted as the probability that a random mapping $V(F) \rightarrow V(G)$ is a homomorphism. 

Let $\{G_n\}$ be a sequence of simple graphs such that $|V(G_n)| \rightarrow \infty$. We can define a notion of convergence based on homomorpism densities.

\begin{definition}
\label{def:conv}
We say that the sequence $\{G_n\}$ {\em converges} if for every simple graph $F$, the sequence $\{ t(F,G_n)\}$ converges. 
\end{definition}

This definition of convergence is non-trivial only for dense graphs, {\it i.e.}~for graph sequences $\{G_n\}$ with the property that $|E(G_n)|= \Omega(|V(G_n)|^2)$. When $\{G_n\}$ consists of sparse graphs, then for all graphs $ F$ with at least one edge, $t(F,G_n) \rightarrow 0$. 

As shown in \cite{borgsI2008},  the notion of convergence of graph sequences  is closely connected to a certain metric space described as follows: Let $\mathcal{W}_0$ denote the set of all measurable functions  $w: [0, 1]^2\rightarrow [0, 1]$ which are symmetric, {\it i.e.} $w(x,y)=w(y,x)$ for every $x,y\in[0,1]$. The elements of ${\mathcal W}_0$ are called {\em graphons}. 
We also denote by $\mathcal W$ the space of all the bounded symmetric measurable functions from $[0,1]^2$ to ${\mathbb R}$.
We can extend the definition of homomorphism densities to $\mathcal{W}$ as follows. For each function $w\in{\mathcal W}$, let

\begin{equation}
\label{wdensity}
t(F,w) = \int_{[0,1]^k} \prod_{ij \in E(F)} w(x_i,x_j)dx_1 \ldots dx_k,
\end{equation}
where $V(F)=\{1,2,\ldots,k\}$.

A simple  graph $G$, with vertex set  $V(G) = \{1, 2, \ldots , n\}$ and adjacency matrix $A$, can be represented by a function $w_G \in \cal{W}_0$, which takes values in $\{ 0,1\}$. 
Split the interval $[0, 1]$ into $n$ equal intervals $I_1, I_2 \ldots, I_n$. Now for $(x,y)  \in I_i \times I_j $, let

\begin{eqnarray}
\label{grafunct}
w_G(x,y) &=& \left\{ \begin{array}{ll}
                                                  A_{i,j} & \mbox{for} \ i \neq j \\
                                                  1        & \mbox{for} \ i = j 
                                                  \end{array} \right ..
                                                   \end{eqnarray}
Our definition of $w_G$ differs slightly from that given in \cite{lovaszszegedy2006} since we give the diagonal blocks $I_i\times I_i$ value one, not zero.  The advantage of this choice becomes apparent when we discuss  ``diagonally increasing''  functions. It is a convenience and is not essential for the results.

Note that a graph can be represented by many different functions $w_G$. Each labelling of the vertices of $G$ results in a permutation of the rows and columns of the adjacency matrix, and leads to a trivially different function. Since a graph represents an entire isomorphism class, we need to introduce an equivalent notion for functions in ${\cal W}$. 
Recall that a map $\phi:[0,1]\rightarrow [0,1]$ is {\em measure-preserving} if for every measurable set $X\subseteq [0,1]$, the pre-image $\phi^{-1}(X)$ is measurable with the same measure as $X$.
Let $\Phi$ be the set of all invertible maps $\phi:[0,1]\rightarrow[0,1]$ such that both $\phi$ and its inverse are measure-preserving. Any $\phi\in \Phi$ acts on a function $w\in {\cal W}$ by transforming it into a function $w^\phi$, where $w^{\phi}(x,y)=w(\phi(x),\phi(y))$.

The notion of the convergence of a graph sequence can be better understood if ${\mathcal W}$ is equipped with a distance derived from the {\em cut-norm}, introduced in \cite{FriezeKanan99} and defined as follows: For all $w \in \cal{W}$,

\begin{equation}
\label{cutnorm} 
\lVert w \rVert_{\square} = \underset{S,T \subset [0,1]} \sup \Bigl \lvert \int_{S \times T} w(x,y) dxdy \Bigr \rvert,
\end{equation}
 where $S$ and $T$ are measurable subsets of $[0,1]$.
We then define the {\em cut-distance} of two functions $w_1$ and $w_2$ in ${\mathcal W}$ by
\begin{equation}
\label{cutdist} 
\delta_{\Box}(w_1,w_2)=\inf_{\phi\in\Phi} \lVert w_1 - w_2^\phi \rVert_{\square} = \inf_{\phi\in\Phi}\underset{S,T \subset [0,1]} \sup \Bigl \lvert \int_{S \times T} ( w_1 - w_2^\phi ) \Bigr \rvert .
\end{equation}
This yields the definition of the cut-distance of two (unlabelled) graphs $G$ and $G'$, defined as
\begin{equation}
\label{cutdistgraph} 
\delta_{\Box}(G,G')=\delta_{\Box}(w_G,w_{G'}).
\end{equation}
The choice of term ``distance'' rather than ``metric'' is due to the fact that $\delta_\Box(G,G')$ can be zero for different graphs $G$ and $G'$, for example when $G'$ is the $k$-fold blow-up of $G$
(see \cite{borgsI2008} for more details).

It is shown in Theorem 3.8 of \cite{borgsI2008} that a graph sequence $\{ G_n\}$ converges whenever the corresponding sequence of functions $w_{G_n}$ is $\delta_\Box$-Cauchy. Moreover, to a convergent graph sequence $\{G_n\}$, one assigns a ``limit object'' represented by a function $w\in {\mathcal W}_0$  (not necessarily integer-valued, or corresponding to a graph). More precisely, for every convergent sequence $\{G_n\}$, there exists  $w$ in ${\mathcal W}_0$ such that the homomorphism densities $t(F,G_n)$ converge to the homomorphism densities $t(F,w)$ for every finite simple graph $F$. If this is the case, we say $\{G_n\}$ converges to $w$, and write $G_n\rightarrow w$.
Such a  function $w$ encodes the common structure of the graphs of the sequence. For more details, see  \cite{lovaszszegedy2006}.
In this paper, we use the following characterization of convergent graph sequences which is given in \cite{borgsI2008}.

\begin{theorem}\cite{borgsI2008}
A sequence $\{G_n\}$ converges to a function $w$ in ${\mathcal W}_0$ if and only if $\delta_\Box(w_{G_n},w)\rightarrow 0$. Furthermore, if this is the case, and $\|V(G_n)\|\rightarrow \infty$, then there is a way to label 
the vertices of the graphs $G_n$ such that $\|w_{G_n}-w\|_\Box\rightarrow 0$.
\end{theorem}

The limit object of a convergent graph sequence is unique up to measure-preserving transformations. Namely $w$ and $w'$ are limits of a convergent graph sequence $\{G_n\}$ if and only if $w^\phi=w'^\psi$ almost everywhere for some measure-preserving maps $\phi,\psi:[0,1]\rightarrow [0,1]$ (or equivalently whenever $\delta_\Box(w,w')=0$).
Note that cut-distance does not define a metric on ${\mathcal W}$, as two different functions can have $\delta_{\Box}$-distance zero. 
We say two functions $w',w\in{\mathcal W}_0$ are {\em equivalent}, and we write $w'\approx w$, if $\delta_\Box(w',w)=0$.
Identifying equivalent functions $w$ and $w'$ in ${\mathcal W}$, we consider the cut-distance as a metric on  the quotient space ${\mathcal W}/\approx$, denoted by  $\widetilde{{\mathcal W}}$. Similarly, we define the set $\widetilde{{\mathcal W}_0}$ of {\em unlabelled graphons}.
It was shown in \cite{lovasz-szegedy-GAFA2007} that  $\widetilde{{\mathcal W}_0}$ 
is in fact a compact metric space.

Finally, given any function $w\in {\mathcal W}_0$, and integer $n$, we define the random graph $G(n,w)$ to be the probability space of graphs on vertex set $\{ 1,2,\dots ,n\}$ obtained through the following stochastic process: Each vertex $j$ receives a value $x_j$, drawn independently and uniformly at random from $[0,1]$. For each pair $i<j$, independently, vertices $i$ and $j$ are then linked with conditional probability $w(x_i,x_j)$. In \cite{lovaszszegedy2006}, it is shown that, asymptotically almost surely, for any finite graph $F$, the homorphism density $t(F,G)$ for a graph $G$ produced by $G(n,w)$ is arbitrarily close to $t(F,w)$. Thus, a graph sequence $\{ G_n\}$, where for each $n$, $G_n$ is produced by $G(n,w)$, almost surely converges to $w$.

\section{Linear embeddings and the parameter $\Gamma^*$}\label{Gamma*}

In this section, we  will define a graph parameter $\Gamma^*$ which is zero precisely when the graph is a unit interval graph, or one-dimensional geometric graph, and thus has a natural linear embedding. In subsequent sections we will then introduce a related parameter $\Gamma$ which applies to functions in ${\cal W}_0$. Using graph limits, we will show a close relationship between the two parameters, especially when applied to convergent graph sequences.

First, we need precise definitions of the concepts discussed in the introduction. Following the convention, see for example \cite{JanLucRuc}, we use both {\sl random graph} and {\sl random graph model} to denote a discrete probability space where the sample space is the set of all graphs on a given vertex set. The notation $u\sim v$ signifies \lq\lq $u$ is adjacent to $v$''. The {\sl link probability} for a given pair of vertices $u,v$  is the probability of the event $u\sim v$. 

Given a convex region $S\subseteq\Rrr^k$ equipped with a metric $d$ derived from one of the $L_p$ norms, we define a symmetric function $f:S\times S\rightarrow [0,1]$ to be a {\em spatial link function} if for every $a\in[0,1]$ and for every $x\in S$, the region $R_a(x)=\{ y: f(x,y)\geq a\}$ is a convex set containing $x$. 
Thus, if we move a point $y$ away from a given point $x$ along a ray starting at $x$, then $f(x,y)$ decreases as the distance from  $x$ increases. This does not mean that $f(x,y)$ is always decreasing as the distance $d(x,y)$ is increasing, however. For example, if $S=[0,1]$, one can define a spatial link function $f$ as follows: 
\[
f(x,y)=\left\{\begin{array}{ll}1-|x-y|&\text{if }x+y\geq 1,\\ x+y-|x-y|&\text{otherwise.}\end{array}\right.
\]
Then for $f(\frac12,\frac12+\delta)=1-\delta$, and $f(\frac12,\frac12-\delta)=1-2\delta$. In both cases, the link probability decreases as $\delta$ increases, but the rate is different for values on different sides of $\frac12$.

Let $k$ be a positive integer, and  $S$ be a convex region in $\Rrr^k$. Let $d$ denote a metric derived from one of the $L_p$ norms on $S$.
Fix $n\in \Nnn$. For a  spatial link function $f$ and a probability measure $\mu$ on $S$, we define a {\sl spatial random graph} $SG(S,d,f,\mu ,n)$   to be a random graph with vertex set $\{ 1,2,\dots ,n\}$ formed according to the following process. Each vertex $j$ receives a value $x_j$, drawn from $S$ according to the probability distribution given by $\mu$.  For each pair $i<j$, independently,  vertices $i$ and $j$ are then linked with a conditional probability which equals $f(x_i,x_j)$. 

\begin{definition}
\label{spatial}
A random graph on the vertex set $\{ 1,2,\dots ,n\}$ has a {\sl spatial embedding} into a given metric space $(S,d)$ if there exist a probability distribution $\mu$ and a link probability function $f$ so that the random graph corresponds to the spatial random graph $SG(S,d,f,\mu ,n)$ ({\it i.e.}~gives the same probability distribution on the sample space of all graphs with  vertex set $\{ 1,2,\dots ,n\}$).
A {\sl linear embedding} is a spatial embedding into $(\Rrr,|\cdot |)$.
\end{definition}

The notion of spatial embedding can be seen as a \lq\lq fuzzy" version of a random geometric graph.
A graph $G$ is called a \emph{geometric graph} on a bounded region $S\subseteq \Rrr^k$ with metric $d$ if there exists an embedding $\pi$ of the vertices of $G$ in $S$, and a threshold value $r>0$, such that for every two vertices
$u$ and $v$ of $G$, $u$ is adjacent to $v$ if and only if $d(\pi(u),\pi(v))\leq r$.   Geometric graphs have been studied extensively; see for example \cite{Clark90,Dall02,Pen03}. The random geometric graph $RG(S,n,r)$ is the geometric graph which results if the embeddings of the vertices are chosen randomly from $S$.
Random geometric graphs clearly have a spatial embedding. Link probabilities in this case can only be 1 or 0. Precisely, the spatial link function $f$ is given by $f(x,y)=1$ if $d(x,y)\leq r$, and $f(x,y)=0$ otherwise. 
For all $a\in [0,1]$, $R_a(x)$ equals the closed ball around $x$ of radius $r$, so clearly $f$ is a spatial link function. 
In this paper, we restrict ourselves to geometric graphs on the one-dimensional space $(\Rrr,|\cdot |)$ and will refer to these as one-dimensional geometric graphs. 

We introduce first a graph parameter $\Gamma^*$, which characterizes geometric graphs in $(\Rrr,|\cdot |)$.  One-dimensional geometric graphs are also known as unit interval graphs. The correspondence becomes clear if we associate each vertex $u$ of a one-dimensional geometric graph with the interval $[\pi( u)-\frac{1}2,\pi(u)+\frac{1}2]$, where $\pi$ is the geometric embedding. 
(We can always assume,  without loss of generality,  that $r=1$.) Now  vertices $u$ and $v$ are adjacent precisely when the associated intervals overlap. 

It is well known that unit interval graphs are characterized by  the consecutive $1$s property of
the vertex-clique matrix (see \cite{gardi}). Restating this property, it follows that a graph $G$ is one-dimensional geometric if and only if there exists an ordering $\prec$ on the vertex set of $G$ such that
\begin{equation}\label{eq:1-dim-geom}
\forall {v,z,w\in V(G)},\, v\prec z\prec w \ \mbox{ and }\ v\sim w\Rightarrow   \ z\sim v \ \mbox{ and }\ z\sim w.
\end{equation}
%
%

To be self-contained, we present a direct proof below.

 \begin{proposition}
 \label{prop:1DiffOrder}
 A graph $G$ is a one-dimensional geometric graph (unit interval graph) if and only if there exists an ordering $\prec$ on $V(G)$ that satisfies  (\ref{eq:1-dim-geom}).
 \end{proposition}
 \begin{proof}
 The forward direction is clear. To prove the converse, we proceed by  induction. Suppose that for  every graph $G$ with $k<n$ vertices, if $V(G)$ satisfies 
 (\ref{eq:1-dim-geom}) for an ordering $\prec$, then there exists a linear embedding $\pi$ of vertices of $G$, with the additional conditions that $\pi$ is injective, and that the distance between adjacent vertices is strictly less than one. Also, we assume that the embedding respects the ordering $\prec$, so $u\prec v$ implies that $\pi(u)< \pi(v)$.
  
Suppose that $G$ is a graph with $n$ vertices, and there exists an ordering $\prec$ on vertices of $G$ which satisfies (\ref{eq:1-dim-geom}).

Let  $\{v_1,\ldots,v_n\}$ be the vertices of $G$ labeled such that $v_i\prec v_j$ whenever $i<j$. 
 The ordering $\prec$ restricted to $V(G)\setminus\{v_n\}$ satisfies Condition (\ref{eq:1-dim-geom}) for $G-v_n$. Thus, by the induction hypothesis,  $G-v_n$ 
has a linear embedding $\pi$ of $V(G)\setminus\{v_n\}$ into the real line which satisfies the additional conditions. Suppose that $m$ is the smallest index such that $v_m$  is adjacent to $v_n$. Let $\ell=\max\{\pi(v_{n-1}),\pi(v_{m-1})+1\}$, and consider the interval $(\ell,\pi(v_m)+1)$. By the induction hypothesis, $\pi(v_{m-1})<\pi( v_{m})$, and, since $v_m$ and $v_n$ are adjacent, so are $v_m$ and $v_{n-1}$, and thus $\pi(v_{n-1})<\pi(v_m)+1$. This implies that $\ell <\pi(v_m)+1$, and thus
the interval is non-empty. Moreover, every point in the interval has distance greater than  one to all embeddings of non-neighbours of $v_n$, and distance less than one to all embeddings of neighbours of $v_n$.  Therefore, choosing $\pi(v_n)$ in this interval results in a linear embedding of $V(G)$ with the desired properties, and we are done.
\end{proof}

Using Condition (\ref{eq:1-dim-geom}), we define a parameter $\Gamma^*$ on graphs which identifies the one-dimensional geometric graphs. 
Let $G$ be a graph with a linear order $\prec$ on its vertices. For every $v\in V(G)$, we define the {\em down-set}  $D(v)$ and the {\em up-set} $U(v)$  of $v$ as follows:
 $$D(v)=\{x\in V(G): x\prec v\}\ \mbox{ and }\ U(v)=\{x\in V(G):v\prec x\}.$$
 For every vertex $v$,  the collection of all the neighbours of $v$ is denoted by $N(v)$.
\begin{definition}
Let $A\subseteq V(G)$, and $\prec$ be a linear order of the vertex set of $G$. We define,
\begin{eqnarray*}
\Gamma^* (G,\prec, A)&=&\frac{1}{|V(G)|^3}\sum_{u\prec v}{\big[} 
|N(v)\cap A\cap D(u)|-|N(u)\cap A\cap D(u)|{\big]}_+\\
&+&\frac{1}{|V(G)|^3}\sum_{u\prec v}{\big[}
|N(u)\cap A\cap U(v)|-|N(v)\cap A\cap U(v)|{\big ]}_+,
\end{eqnarray*}
where 
$$[x]_+=
\left\{
\begin{array}{cc}
 x & \mbox{ if } x>0  \\
 0 & \mbox{ otherwise}     
\end{array}
\right..
$$
We also define 
\[\Gamma^*(G,\prec)=\max_{A\subseteq V(G)}\Gamma^*(G,\prec,A),\] and\[\Gamma^*(G)=\min_{\prec}\Gamma^*(G,\prec),\]where the minimum is taken over all the linear orderings of $V(G)$.
\end{definition}
\begin{proposition}\label{Gamma*-geometric}
A graph $G$ is one-dimensional geometric if and only if $\Gamma^*(G)=0$.
\end{proposition}
\begin{proof}
Let $G$ be a one-dimensional geometric graph, and $A$ be an arbitrary subset of $V(G)$.
Let $\prec$ be a linear  ordering that satisfies Condition (\ref{eq:1-dim-geom}).  Fix an arbitrary pair of vertices  $u\prec v$ of $G$. By Condition (\ref{eq:1-dim-geom}), if  $z$ belongs to 
$N(v)\cap A\cap D(u)$  then $z$ is adjacent to $u$ as well. Thus $|N(v)\cap A\cap D(u)|\leq |N(u)\cap A\cap D(u)|$. Similarly, 
$|N(u)\cap A\cap U(v)|\leq |N(v)\cap A\cap U(v)|$, which implies that $\Gamma^*(G,\prec)=0$. Thus $\Gamma^*(G)=0$.

Conversely, let $G$ be a graph such that $\Gamma^*(G)=0$. Let $\prec$ be the linear order of $V(G)$ such that $\Gamma^*(G,\prec)=0$.  Let $u\prec v$ be an arbitrary pair of adjacent vertices of $G$, and take $z$ so that 
$u\prec z\prec v$. Since 
$\Gamma^*(G,\prec,A)=0$ for all $A\subseteq V(G)$, choosing $A=\{ v\}$ gives that 
$1=|N(u)\cap \{v\}\cap U(z)|\leq |N(z)\cap \{v\}\cap U(z)|$. This implies that $z$ is adjacent to $v$. Similarly, one can show that $z$ is adjacent to $u$. Thus Condition (\ref{eq:1-dim-geom})
is satisfied for $G$, and $G$ is a geometric graph.
\end{proof}

Next, we extend Condition (\ref{eq:1-dim-geom}) to functions in ${\cal W}_0$. The generalization is obtained by considering functions representing graphs, as introduced in the previous section.
Let $G$ be a one-dimensional geometric graph with a linear ordering $\prec$ of its vertices that satisfies Condition (\ref{eq:1-dim-geom}). Let $w_G$ be the function in ${\cal W}_0$ that represents $G$ with respect to the labelling  of $V(G)$ obtained from the linear  ordering $\prec$. 
It follows that $w_G(x,z)=1$ and $x\leq y\leq z$ imply that $w_G(x,y)=1$  and $w_G(y,z)=1$. We generalize this property as follows:
\begin{definition}
\label{def:di}
A function $w\in {\cal W}$ is \emph{diagonally increasing} if for every $x, y,z\in [0,1]$, we have
\begin{enumerate}
\item $x\leq y\leq z \Rightarrow w(x,z)\leq w(x,y),$
\item $ y\leq z \leq x \Rightarrow w(x,y)\leq w(x,z)$.
\end{enumerate}
A function $w$ in ${\cal W}$ is diagonally increasing almost everywhere if there exists a diagonally increasing function $w'$ which is equal to $w$ almost everywhere.
\end{definition}

Combining definitions \ref{spatial} and \ref{def:di}, it is clear that a symmetric function $w$ is a spatial link function on $[0,1]$ if and only if $w$ is diagonally increasing. 
In the following remark, we show that a $w$-random graph has a ``reasonable'' linear embedding whenever $w$ is equivalent to a diagonally increasing function. 

\begin{remark}
Note that the random graphs $G(n,w)$ and $G(n,w')$ are the same, {\it i.e.} they are identical as probability distributions, if $w\approx w'$. 
To see this, let $\Pr_w(F)$  denote the probability assigned to a simple graph $F$ on vertex set $\{ 1,2,\dots ,\}$ in $G(n,w)$. Clearly,
\begin{eqnarray*}
\Pr_w(F)=\int \prod_{i\sim j} w(x_i,x_j)\prod_{k\not\sim l} (1-w(x_k,x_l))=\sum_{F'} (-1)^{|e(F')|-|e(F)|} t(F',w),
\end{eqnarray*}
where the sum is taken over all graphs $F'$ on vertex set $\{1,2,\ldots, n\}$ which contain $F$ as their subgraph. 
Our claim clearly follows from 
Corollary 3.10 of \cite{borgsI2008}, which we state below:
\begin{quote}
For two graphons $w$ and $w'$ we have $\delta_\Box(w,w')=0$ if and only if $t(F,w)=t(F,w')$ for every simple graph $F$. 
\end{quote}
Thus, if $w$ is equivalent to a diagonally increasing function, then for any integer $n>1$, the random graph $G(n,w)$ has a linear embedding. 

The converse is also true, under certain conditions. Namely, suppose $G(n,w)$ has a linear embedding
$SG([0,1],|\cdot|,f,\mu,n)$. Also suppose that  $\mu$ is a continuous probability distribution ({\it i.e.} absolutely continuous with respect to Haar measure), that assigns nonzero measures to  open intervals in $[0,1]$. Let $F$ be the cumulative distribution function of $\mu$ on $[0,1]$. Then, if $x$ is sampled uniformly from $[0,1]$, $F(x)$ is sampled according to $\mu$. 
Let $w'(x,y)=f(F(x),F(y))$, where $f$ is the spatial link function. An argument similar to our previous discussion implies that for every simple graph $H$, the densities $t(H,w)$ and $t(H,w')$ are the same. Thus,
$\delta_\Box(w,w')=0$. Moreover, $w'$ is diagonally increasing, since $F$ is increasing and $f$ is a spatial link function.  
Therefore $w$ is equivalent to a  diagonally increasing function. 

Clearly, a graph is a one-dimensional geometric graph if and only if it has a function representative in ${\cal W}_0$ which is diagonally increasing. (Remember that we assume the function representative to have all blocks on the diagonal equal to 1.)  Indeed, the function representative will be the function $w_G$ where the vertices are ordered according to a linear ordering that satisfies Condition (\ref{eq:1-dim-geom}). More important is the connection between diagonally increasing functions and linear embeddings, which follows in the next section.
\end{remark}

\section{The parameter $\Gamma$ on ${\cal W}$}
\label{Gamma}
Next, we introduce a parameter $\Gamma$ which generalizes the graph parameter $\Gamma^*$ to functions in ${\cal W}$. We will see that $\Gamma$ identifies the diagonally increasing functions.

\begin{definition}
\label{def:Gamma}
Let $\cal A$ denote the collection of all measurable subsets of $[0,1]$.
Let $w$ be  a function in ${\cal W}$, and $A\in \cal A$.  We define
\begin{eqnarray*}
\Gamma (w,A)
&=&\int\int_{y<z}{\big[} \int_{x\in A\cap [0,y]} \left( w(x,z)-w(x,y) \right)dx{\big]}_+ dy dz\\
&+&\int\int_{y<z}{\big[} \int_{x\in A\cap [z,1]}\left( w(x,y)-w(x,z)\right)dx{\big ]}_+ dy dz.\\
\end{eqnarray*}
Moreover, $\Gamma (w)$ is defined as 
$$\Gamma (w)=\sup_{A\in{\cal A}}\Gamma (w,A),$$ 
where the supremum is taken over all the measurable subsets of $[0,1]$.
\end{definition}

It follows directly from the definitions that any function $w\in\cal W$ which is almost everywhere diagonally increasing has $\Gamma(w) =0$. The converse also holds, as is stated in the following proposition. 

\begin{proposition} \label{prop:Gamma0iffDI}
Let $w$ be a function in ${\cal W}$. The function $w$ is diagonally increasing almost everywhere if and only if $\Gamma (w)=0$.
\end{proposition}

Before we give the proof, we introduce some notations which will be used later.
%
Let $w\in {\cal W}_0$, and $A$ and $B$ be measurable subsets of $[0,1]$. We define $\tilde{w}(A,B)$ to be the average of $w$ on $A\times B$, {\it i.e.}
$$\tilde{w}(A,B)=\frac{1}{\mu(A) \mu(B)}\int_{A\times B}w(x,y)dxdy,$$
where $\mu$ is the Lebesgue measure on $[0,1]$.
 Let $n$ be a positive integer. For each $0\leq i\leq n-1$, let $I_i=[\frac{i}{n},\frac{i+1}{n}]$. We define the symmetric functions $w_n$, $w_n^+$, and $w_n^-$ on $[0,1]^2$ as follows.
 \begin{eqnarray*}
w^n_{i,j}&=& \tilde{w}(I_i,I_j) \mbox{ for } 0\leq i,j\leq n-1,\\
w_n(x,y)&=&w^n_{i,j} \mbox{ if } (x,y)\in I_i\times I_j,\\
 w_n^-(x,y)&=&\left\{
\begin{array}{cc}
 w^n_{i-1,j+1} &   \mbox{ if } (x,y)\in I_i\times I_j\ \&\ 1\leq  i\leq j\leq n-2   \\
 0  & \mbox{ if } (x,y)\in I_0\times I_j   \\
 0  & \mbox{ if } (x,y)\in I_i\times I_{n-1},   \\
\end{array}\right.\\
 w_n^+(x,y)&=&\left\{
\begin{array}{cc}
w^n_{i+1,j-1} &   \mbox{ if } (x,y)\in I_i\times I_j\ \&\  i\leq j-2   \\
1 & \mbox{ if } (x,y)\in I_i\times I_i   \\
1  & \mbox{ if } (x,y)\in I_i\times I_{i+1} .  \\
\end{array}\right.
 \end{eqnarray*}
Let $A$ and $B$ be subsets of $[0,1]$. We say $A\leq B$ if every $a$ in $A$ is smaller than or equal to every $b$ in $B$.

We now give the proof of Proposition \ref{prop:Gamma0iffDI}.
This proof is inspired by the proof of Lemma 4.6 of \cite{bollobas2011}. However, we include the proof to make the paper self-contained.
\begin{proof}[Proof of Proposition \ref{prop:Gamma0iffDI}]
Clearly, if $w$ is diagonally increasing almost everywhere then $\Gamma (w)=0$.  
We now prove  the other direction. First, let us assume that $w$ is a function in ${\cal W}_0$ with $\Gamma (w)=0$. 
Let $A$, $B$, and $C$ be measurable  subsets of $[0,1]$ such that $C\leq A\leq B$. Since $\Gamma (w)=0$, 
for almost every $y\in A$ and almost every $z\in B$, 
\begin{equation}\label{eq1}
\int_{x\in C} w(x,z) dx\leq  \int_{x\in C}w(x,y)dx.
\end{equation}
Taking repeated integrals of both sides of Equation (\ref{eq1}) over $A$ and then  $B$ and then dividing by $\mu(A)$, we conclude that
\begin{equation}\label{eq2}
\int_{C\times B} w(x,z) dxdz\leq  \frac{\mu(B)}{\mu(A)}\int_{C\times A}w(x,y)dxdy.
\end{equation}
 %
 Similarly, one can show that for subsets $A$, $B$, and $C$ of $[0,1]$ with $A\leq B\leq C$, we have
 \begin{equation}\label{eq3}
 \int_{A\times C} w(x,y) dydx\leq  \frac{\mu(A)}{\mu(B)}\int_{B\times C}w(x,z)dzdx.
 \end{equation}
 Applying the above inequalities to the sets $I_i$, we have that for every $(x,y)\in [0,1]^2$,  $w_n^-(x,y)\leq w_n(x,y)\leq w_n^+(x,y)$.  
Now let  $A$ and $B$ be  measurable subsets of $[0,1]$. From Equations (\ref{eq2}) and (\ref{eq3}) it follows that, if $0\leq i\leq j-2\leq n-3$, then
\begin{eqnarray*}
\int_{(A\cap I_i)\times (B\cap I_j)} w(x,y)dxdy &\leq & \frac{\mu(B\cap I_j)}{\mu(I_{j-1})}
\int_{(A\cap I_i)\times I_{j-1}} w(x,y)dxdy\\
&\leq & \frac{\mu(A\cap I_i)\mu(B\cap I_j)}{\mu(I_{i+1})\mu(I_{j-1})}\int_{I_{i+1}\times I_{j-1}}w(x,y)dxdy\\
&=& \mu(A\cap I_i)\mu(B\cap I_j) w^n_{i+1,j-1}.
\end{eqnarray*}
Thus,
\begin{eqnarray}\label{eq:4.2}
\int_{(A\cap I_i)\times (B\cap I_j)} w(x,y)dxdy\leq \int_{(A\cap I_i)\times (B\cap I_j)} w_n^+(x,y)dxdy.
\end{eqnarray}
By definition of $w_n^+$, similar inequalities hold trivially for the cases where $i=j-1$ or $i=j$. 
Finally, using the fact that $w$ is symmetric, we conclude that (\ref{eq:4.2}) holds for every $i$ and $j$. Therefore,
 \begin{eqnarray*}
\int_{A\times B}w(x,y)dxdy
&=&\sum_{i,j=0}^{n-1}\int_{(A\cap I_i)\times (B\cap I_j)} w(x,y)dxdy\\
&\leq& \sum_{i,j=0}^{n-1}\int_{(A\cap I_i)\times (B\cap I_j)} w^+_n(x,y)dxdy\\
 &=&\int_{A\times B}w_n^+(x,y)dxdy.
 \end{eqnarray*}
 Moreover, since measurable subsets of $[0,1]^2$ can be approximated in measure by finite unions of disjoint rectangles, we get
 $$\int_{E} w(x,y) dxdy\leq  \int_{E}w^+_n(x,y)dxdy,$$
 for every measurable subset $E$ of $[0,1]^2$. Thus, $w\leq w_n^+$ (and similarly $w_n^-\leq w$) almost everywhere in $[0,1]^2$.
Therefore, 
$$\|w-w_n\|_1\leq \|w_n^+-w_n^-\|_1=\int_{[0,1]^2}(w_n^+-w_n^-)(x,y)dxdy.$$
By the definitions of $w_n^+$ and $w_n^-$, we have $\int_{I_i\times I_j} w_n^-(x,y) dxdy= \int_{I_{i-2}\times I_{j+2}}w^+_n(x,y)dxdy$ for every pair $i,j$ satisfying $2\leq i\leq j-1\leq\ n-4$. Moreover, 
$w_n^+,w_n^-\in{\mathcal W}_0$. Thus, 
$$\|w-w_n\|_1\leq \frac{8}{n}.$$
Using the Borel-Cantelli lemma, we conclude that the sequence $\{w_{2^n}\}_{n\in {\mathbb N}}$ converges to $w$ almost everywhere in $[0,1]^2$, {\it i.e.}
$\psi:=\limsup_{n\in {\mathbb N}}w_{2^n}=w \ \mbox{ almost everywhere}.$
Finally, by Equations (\ref{eq2}) and (\ref{eq3}), each $w_n$ is a diagonally increasing function. Therefore, $\psi$ is diagonally increasing as well. This proves the converse for the case where $w\in{\mathcal W}_0$. 

Now let $w$ be an element of ${\mathcal W}$ such that $\Gamma(w)=0$. Define the new symmetric function $w'$ to be $w'=\frac{w-a}{b-a}$, where $a$ (respectively $b$) is a lower bound 
(respectively upper bound) for $w$. Then $w'\in{\mathcal W}_0$ and $\Gamma(w')=0$. Therefore, by the previous part of the proof, we have that $w'$ is diagonally increasing almost everywhere. Hence, $w$ is diagonally increasing almost everywhere as well. 
\end{proof}

\section{Parameters $\Gamma^*$ and $\Gamma$ asymptotically agree on graphs}\label{relation}

A graph $G$ can be represented as a function $w_G\in {\cal W_0}$, but it is not necessarily true that $\Gamma^*(G)=\Gamma ( w_G)$, even when the representation $w_G$ is obtained by using the ordering of the vertices that achieves $\Gamma^*(G)$. This is due to the fact that a set $A$ which determines the value of $\Gamma (w)$ does not have to be consistent with the partition of $[0,1]$ into $n$ equal-sized parts on which $w_G$ is defined. However,  we show  that $\Gamma^* (G)$ and $\Gamma (w_G)$, computed using the same ordering of the vertices, are asymptotically equal. This result follows as a corollary from the following theorem.

%
%
%


\begin{theorem}\label{thm:approximation-of-gamma}
Let $n\in{\mathbb N}$. Let $w\in {\cal W}_0$ be a function which is measurable with respect to the product algebra  $\cal {A}_n^*\times \cal {A}_n^*$, where the algebra  $\cal {A}_n^*$ is  generated by the intervals
$\{I_i: 0\leq i\leq n-1\}$. Then
$$\Gamma (w)=\sup_{A\in \cal A}\Gamma (w,A)=\max_{ A\in \cal A_n^*}\Gamma (w,A)+\cal{O}(\frac{1}{n}).$$

\end{theorem}

\begin{proof}
Let $n\in\mathbb{N}$ and $w\in {\cal W}_0$ be as above. Note that  $w$ is constant on the rectangles $I_i\times I_j$, since it is measurable with respect to the product algebra $\cal {A}_n^*\times \cal {A}_n^*$.
For each $i,j\in\{0,\ldots,n-1\}$, let $w(x,y)=a_{ij}$ whenever $(x,y)\in I_i\times I_j$.
 Fix $A\in \cal A$, and let
$\beta_k=\mu(A\cap I_k)$ for every $0\leq k\leq n-1$. The expression for $\Gamma(w,A)$ as given in Definition \ref{def:Gamma} can now be simplified.

Consider $y<z$ so that $y\in I_i$ and $z\in I_j$. If $i=j$, then for all $x$, $w(x,z)=w(x,y)$, so $\big[ \int_{x\in A\cap [0,y]}\left(w(x,z)-w(x,y) \right)dx\big]_+=0$. If $0\leq i<j\leq n-1$, then

\begin{eqnarray*}
&&\big[\int_{x\in A\cap [0,y]} \left(w(x,z)-w(x,y) \right)dx\big]_+\\
&=&\big[\sum_{k=0}^{i-1}{\int_{A\cap I_k}}(a_{kj}-a_{ki})dx+\int_{A\cap I_i\cap [0,y]}(a_{ij}-a_{ii})dx\big]_+
\\
&=&\big[\sum_{k=0}^{i-1}\mu(A\cap I_k)(a_{kj}-a_{ki})+\mu(A\cap I_i\cap[0,y])(a_{ij}-a_{ii})\big]_+
\\
&\leq &\left(\big{[}\sum_{k=0}^{i-1}\beta_k(a_{kj}-a_{ki})\big{]}_+ +\frac{2}{n}\right).
\end{eqnarray*}
In the last step, we use the inequality $[x+y]_+\leq [x]_++[y]_+$, and the fact that  $w$ is bounded by $1$, so $|\mu(A\cap I_i\cap[0,y])(a_{ij}-a_{ii})|$ is at most $\frac{2}{n}$. 

Similarly, we have that
\begin{eqnarray*}
&& {\big[} \int_{x\in A\cap [z,1]}\left(w(x,y)-w(x,z)\right)dx{\big ]}_+
\\
&=&\big[ \sum_{k=j+1}^{n-1}\mu(A\cap I_k)(a_{ki}-a_{kj})+\mu(A\cap I_j\cap[z,1])(a_{ji}-a_{jj}){\big]}_+
\\
&\leq & \left( {\big[}\sum_{k=j+1}^{n-1}\beta_k(a_{ki}-a_{kj}) {\big]}_+ +\frac{2}{n}\right).
\end{eqnarray*}

Using this, we can bound $\Gamma(w,A)$:

\begin{eqnarray*}
\Gamma  (w,A) &\leq &
\sum_{0\leq i<j\leq n-1}\int_{y\in I_i}\int_{z\in I_j}\left(\big{[}\sum_{k=0}^{i-1}\beta_k(a_{kj}-a_{ki})\big{]}_++\frac{2}{n}\right)dy dz
\\
&+&\sum_{0\leq i<j\leq n-1}\int_{y\in I_i}\int_{z\in I_j}\left(\big{[}\sum_{k=j+1}^{n-1}\beta_k(a_{ki}-a_{kj})\big{]}_++\frac{2}{n}\right)dy dz
\\
&=&\sum_{0\leq i<j\leq n-1}\frac{1}{n^2}\big{[}\sum_{k=0}^{i-1}\beta_k(a_{kj}-a_{ki})\big{]}_++\frac{n-1}{n^2}
\\
&+&\sum_{0\leq i<j\leq n-1}\frac{1}{n^2}\big{[}\sum_{k=j+1}^{n-1}\beta_k(a_{ki}-a_{kj})\big{]}_++\frac{n-1}{n^2}.\\
\end{eqnarray*}

Now define,
\begin{eqnarray*}
g_w(A)&=&g_w(\beta_0,\ldots,\beta_{n-1})\\
&=& \sum_{0\leq i<j\leq n-1}\frac{1}{n^2}\left(
\big{[}\sum_{k=0}^{i-1}\beta_k(a_{kj}-a_{ki})\big{]}_+ 
+\big[ \sum_{k=j+1}^{n-1}\beta_k(a_{ki}-a_{kj})\big{]}_+\right).
\end{eqnarray*}
%
Thus,
%
\begin{equation}\label{eq:upper-bound}
\Gamma  (w,A)\leq g_w(A)+\frac{2(n-1)}{n^2}\leq g_w(A)+\frac{2}{n}.
\end{equation}
Similarly, one can use the inequality $[x+y]_+\geq [x]_+-|y|$
 to show that 
\begin{equation}\label{eq:lower-bound}
\Gamma  (w,A)\geq g_w(A)-\frac{2}{n}.
\end{equation}
Since $x\mapsto [x]_+$ is a convex function, $g_w$ is the sum of convex functions, and therefore is itself also convex. Moreover, since $\beta_k\in [0,\frac{1}{n}]$, the function $g_w$ achieves its maximum when each of the coefficients $\beta_k$ is either $0$ or $\frac{1}{n}$. Since $\beta_k=\mu(A\cap I_k)$, this implies that the maximum is achieved when, for each $k$, either $A$ contains $I_k$, or is disjoint from $I_k$.
Hence, $\sup_{A\in\cal A}g_w(A)=\max_{A\in \cal A_n^*}g_w(A)$. 
%
%

Let $A'\in\cal A_n^*$ be such that  $\max_{A\in \cal A_n^*}g_w(A)=g_w(A')$. Then, by Equation (\ref{eq:upper-bound}) and (\ref{eq:lower-bound}) we have,
\begin{eqnarray}\label{eq:1}
\underset{A\in \cal A}\sup \Gamma  (w,A)&\leq&\underset{A\in\cal A}\sup g_w(A)+\frac{2}{n}
=\max_{A\in \cal A_n^*} g_w(A)+\frac{2}{n}\nonumber\\
&=&g_w(A')+\frac{2}{n}
\leq\Gamma  (w,A')+\frac{4}{n}\nonumber\\
&\leq&\max_{A\in\cal A^*}\Gamma  (w,A)+\frac{4}{n}.
\end{eqnarray}
On the other hand, we clearly have
\begin{equation}\label{eq:2}
\max_{A\in\cal A_n^*}\Gamma  (w,A)\leq\underset{A\in\cal A}\sup\Gamma  (w,A),
\end{equation}
completing the proof. 
\end{proof}

\begin{corollary}\label{cor:gamma-graph-function}
Let $G$ be a graph with $n$ vertices, and $w_G$ be the function in ${\cal W}_0$ that represents $G$ with respect to a linear ordering $\prec$ of the vertices of $G$. Then
$$\Gamma^*(G,\prec)=\Gamma (w_G)+\cal O(\frac{1}{n}).$$
\end{corollary}
\begin{proof}
Let $A\in\cal A_n^*$, and define $\tilde{A}=\{0\leq i\leq n-1: I_i\subseteq A\}$. From the proof of Theorem \ref{thm:approximation-of-gamma}, it is easy to observe that $\Gamma^*(G,\prec, \tilde{A})=g_{w_G}(A)$, and
$g_{w_G}(A)-\frac{2}{n}\leq \Gamma (w_G,A)\leq g_{w_G}(A)+\frac{2}{n}.$
Thus,
$$\max_{A\in \cal A_n^*}\Gamma^*(G,\prec,\tilde{A})-\frac{2}{n}\leq \max_{A\in \cal A_n^*}\Gamma (w_G,A)\leq \max_{A\in \cal A_n^*}\Gamma^*(G,\prec,\tilde{A})+\frac{2}{n}.$$
Using Theorem \ref{thm:approximation-of-gamma}, we conclude that $|\Gamma^*(G,\prec)-\Gamma (w_G)|\leq \frac{6}{n}$, and we are done. 
\end{proof}

\section{Continuity of the parameter $\tilde\Gamma$.}\label{continuity}

Our main result, presented in this section, concerns the behaviour of  the parameter $\Gamma^*$ if applied to a converging 
graph sequence $\{ G_n\}$. Using the theory developed in the previous sections, we will show that the sequence $\{ \Gamma^*(G_n)\}$ converges.
Precisely, suppose $\{ G_n\}$ converges to a limit $w\in {\cal W}_0$. Then  there exists a function $w^*$ arbitrarily close to $w$ under the box distance, so that $\lim_{n\rightarrow \infty}\Gamma^*(G_n)$ is arbitrarily close to $\Gamma (w^*)$.

The above follows from the continuity of a related parameter, $\tilde{\Gamma}$, which is defined on ${\cal W}_0$ as the infimum of $\Gamma (w)$ over a set of functions that have box distance zero to each other. The precise definition is given below in Definition \ref{def:tilde}. We first present the following lemmas.


\begin{lemma}\label{lem:w-chi-cut}
Let $w:[0,1]^2\rightarrow[-2,2]$ be a measurable function. Then $\|w\chi\|_\Box\leq 2\sqrt{\|w\|_\Box}$, where 
$$\chi(x,y)=\left\{
\begin{array}{cc}
1  &  x\leq y  \\
 0 &  \mbox{otherwise}
\end{array}
\right ..$$
\end{lemma}
\begin{proof}
Let $\Omega=\{ (x,y): 0\leq x\leq y \leq 1\}$ denote the subset of points above the diagonal in $[0,1]^2$.
Define $k=\lceil\frac{1}{\sqrt{\|w\|_\Box}}\rceil$, which is a positive integer.  Now, we can decompose 
$\Omega$ into $k-1$ rectangles and $k$ triangles as shown in Figure \ref{rectangles}. Precisely, the $i$-th rectangle has width $\frac{1}{k}$ and ranges from $y=\frac{i}{k}$ to $y=1$, and each triangle has base and height equal to $\frac{1}{k}$. By the definition of cut-norm, the integral of $w$ over each of the rectangles is at most $\|w\|_\Box$, in absolute value. Also, each of the triangles has measure $\frac{1}{2k^2}$, and there are $k$ triangles in total. Since $|w|$ is bounded by 2, the integral of $w$ over the triangles is at most $ \frac{1}{2k^2}(2)(k)=\frac{1}{k}$, in absolute value.  
\begin{figure}[ht]
\label{rectangles}
\centerline{\includegraphics[width=0.45\textwidth]{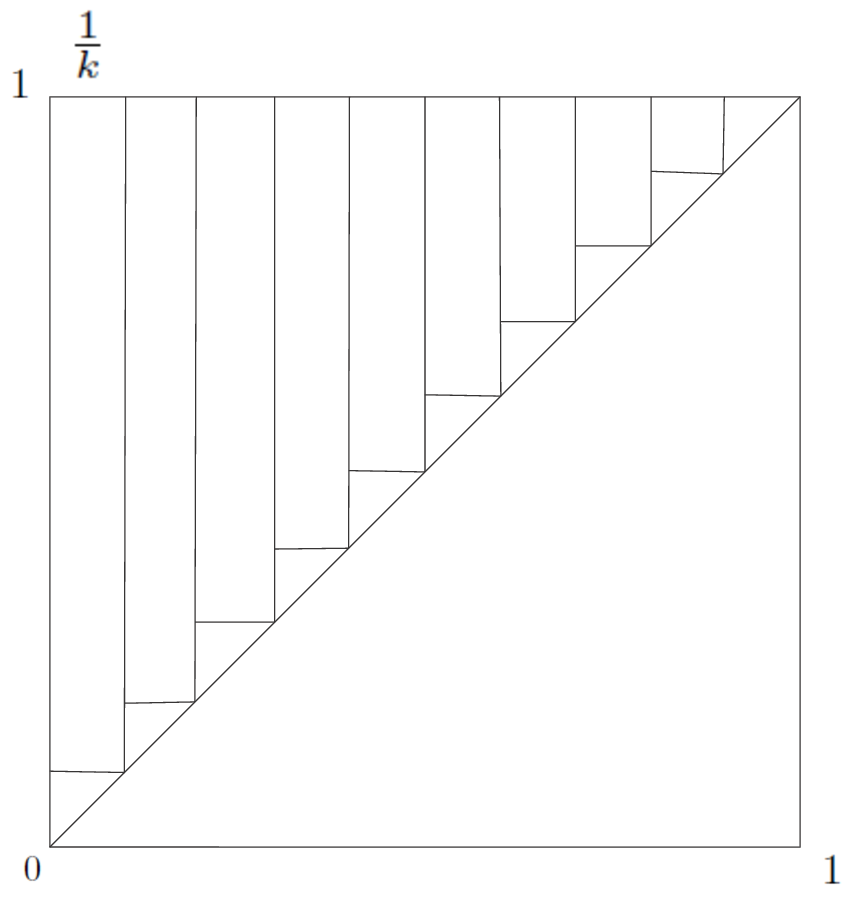}}
\caption{The decomposition of the set of points above the diagonal as used in the proof of Lemma \ref{lem:w-chi-cut}}
\end{figure}
Therefore, we have 
\begin{eqnarray}\label{eq-estimate}
\left|\int_0^1\int_0^1w\chi (x,y)dxdy\right|&\leq& \frac{1}{k}+(k-1)\|w\|_\Box\nonumber\\
&\leq&\sqrt{\|w\|_\Box}+(\frac{1}{\sqrt{\|w\|_\Box}})\|w\|_\Box=2\sqrt{\|w\|_\Box}.
\end{eqnarray}
For arbitrary subsets  $A$ and $B$ of $[0,1]$, let $\chi_{A\times B}$ denote the characteristic function of the subset $A\times B$ of $[0,1]^2$.
Applying (\ref{eq-estimate}) to $ w\chi_{A\times B}$ instead of $w$, we get $|\int_0^1\int_0^1 w\chi_{A\times B}\chi (x,y)dxdy|\leq 2\sqrt{\|w\chi_{A\times B}\|_\Box}\leq 2\sqrt{\|w\|_\Box}$, which proves that $\|w\chi\|_\Box\leq 2\sqrt{\|w\|_\Box}$.
\end{proof}
\begin{lemma}\label{lem:Gamma-cts}
Let $w_1$ and $w_2$ be elements of $\cal W_0$. Then $|\Gamma (w_1)-\Gamma (w_2)|\leq 2\|w_1-w_2\|_\Box+4\sqrt{\|w_1-w_2\|_\Box}$.
\end{lemma}
\begin{proof}
Let
\begin{eqnarray*}
\Gamma _1(w,A)&=&\int\int_{y<z}\left[ \int_{x\in A\cap [0,y]} \left(w(x,z)-w(x,y) \right)dx\right]_+dy dz,\\
\Gamma _2(w,A)&=&\int\int_{y<z}\left[ \int_{x\in A\cap [z,1]}\left(w(x,y)-w(x,z)\right)dx\right]_+dy dz,
\end{eqnarray*}
so $\Gamma(w,A)=\Gamma_1(w,A)+\Gamma_2(w,A)$.
Fix a measurable set $A\in\cal A$. Using again the inequality $[x+y]_+\leq [x]_++[y]_+$, we obtain that 
\begin{eqnarray*}
\Gamma _1(w_1,A)&=&\int\int_{y<z}\left[ \int_{x\in A\cap [0,y]} \left(w_1(x,z)-w_1(x,y) \right)dx\right]_+dy dz\\
&\leq&\int\int_{y<z}\left[ \int_{x\in A\cap [0,y]} \left(w_1(x,z)-w_2(x,z) \right)dx\right]_+dy dz\\
&+&\int\int_{y<z}\left[ \int_{x\in A\cap [0,y]} \left(w_2(x,z)-w_2(x,y) \right)dx\right]_+dy dz\\
&+&\int\int_{y<z}\left[ \int_{x\in A\cap [0,y]} \left(w_2(x,y)-w_1(x,y) \right)dx\right]_+dy dz.
\end{eqnarray*}
Recall that a function on $[0,1]$ attains a value at least as large as the average of the function at some point. Therefore there exists $y_0,z_0\in [0,1]$ such that
\begin{eqnarray*}
\Gamma _1(w_1,A)&\leq&\int_{y_0<z}\left[ \int_{x\in A\cap [0,y_0]} \left(w_1(x,z)-w_2(x,z) \right)dx\right]_+ dz\\
&+&\Gamma _1(w_2,A)\\
&+&\int_{y<z_0}\left[ \int_{x\in A\cap [0,y]} \left(w_2(x,y)-w_1(x,y) \right)dx\right]_+dy \\
&=&\int_{z\in T_1}\int_{x\in A\cap [0,y_0]} \left(w_1(x,z)-w_2(x,z) \right)dx dz\\
&+&\Gamma _1(w_2,A)\\
&+&\int_{y\in T_2}\int_{x\in A\cap [0,y]} \left(w_2(x,y)-w_1(x,y) \right)dx dy,\\
\end{eqnarray*}
where $T_1$ and $T_2$ are the appropriate sets of points which make the associated expressions positive. From the definition of the cut-norm, it then follows that
$$\Gamma _1(w_1,A)-\Gamma _1(w_2,A)\leq \|w_1-w_2\|_\Box+\|(w_1-w_2)\chi\|_\Box.$$
Similarly, by switching $w_1$ and $w_2$, we get $\Gamma _1(w_2,A)-\Gamma _1(w_1,A)\leq \|w_1-w_2\|_\Box+\|(w_1-w_2)\chi\|_\Box,$ which implies that
$$|\Gamma _1(w_1,A)-\Gamma _1(w_2,A)|\leq \|w_1-w_2\|_\Box+\|(w_1-w_2)\chi\|_\Box$$
holds for every  subset $A$. Moreover, one can prove the analogus result for $\Gamma _2$. Thus,
\begin{eqnarray*}
|\Gamma (w_1,A)-\Gamma (w_2,A)|&\leq &|\Gamma _1(w_1,A)-\Gamma _1(w_2,A)|+|\Gamma _2(w_1,A)-\Gamma _2(w_2,A)|\\
&\leq & 2\|w_1-w_2\|_\Box+2\|(w_1-w_2)\chi\|_\Box.
\end{eqnarray*}
Since $\Gamma(w_i)=\sup_A\Gamma(w_i,A)$ for $i=1,2$, it follows that
$$|\Gamma (w_1)-\Gamma (w_2)|\leq 2\|w_1-w_2\|_\Box+2\|(w_1-w_2)\chi\|_\Box.$$
This fact, together with Lemma \ref{lem:w-chi-cut}, finishes the proof.
\end{proof}

We are now ready to prove our continuity result. 
In order to study the limit of the sequence $\{\Gamma^*(G_n)\}$, we need to define the following parameter, which is a generalized notion of $\Gamma$. 
Recall that two functions $u, w\in{\mathcal W}_0$ are equivalent ({\it i.e.} $u\approx w$) precisely when $\delta_\Box(u,w)=0$.

\begin{definition}
\label{def:tilde}
Let $w$ be a bounded function in $\cal W$. We define the new parameter $\tilde{\Gamma}$ to be
$$\tilde{\Gamma}(w):=\inf_{w'\approx w}\Gamma(w')=\inf\{\Gamma(w'): \delta_\Box(w,w')=0\}.$$
\end{definition}

The lemmas above lead to the following theorem, which establishes the continuity of the parameter $\tilde{\Gamma}$ on the space ${\cal W}_0$ with the cut-distance $\delta_\Box$.

\begin{theorem}
\label{thm:Gamma-cts}
Let  $w\in {\cal W}_0$ be the limit of a $\delta_{\Box}$-convergent sequence $\{w_n\}_{n\in {\mathbb N}}$ of functions in ${\cal W}_0$. Then  
 $\{\tilde{\Gamma}(w_n)\}_{n\in {\mathbb N}}$ converges to $\tilde\Gamma (w)$ as $n\rightarrow \infty$.
\end{theorem}

\begin{proof}
By the definition of $\tilde{\Gamma}$, for each positive integer $m$ there exists  an element $u_m\in{\mathcal W}_0$ such that  $\delta_\Box(w,u_m)=0$ and $|\Gamma (u_m)-\tilde\Gamma (w)|\leq\frac{1}{m}$. Fix such a sequence of graphons $\{ u_m\}_{m=1}^{\infty}$.

Fix $m\in {\mathbb N}$. Then 
\[
\delta_\Box(w_n,u_m)=\delta_{\Box}(w_n,w)\rightarrow 0,
\]
as $n$ goes to infinity.
By the definition of cut-distance, this convergence implies that 
there exist maps $\psi_{n}\in\Phi$ such that 
$$\|w_{n}^{\psi_n}-u_m\|_\Box\rightarrow 0\ \mbox{ as }n\rightarrow \infty.$$
 By Lemma \ref{lem:Gamma-cts}   we have,
 $$\Gamma (w_n^{\psi_n}) \rightarrow \Gamma(u_m).$$
 Thus, for every $m\in {\mathbb N}$,
 $$\limsup_{n\in \mathbb N} \tilde{\Gamma} (w_n)\leq 
 \limsup_{n\in \mathbb N} \Gamma (w_n^{\psi_n})=
 \Gamma (u_m)\leq \tilde{\Gamma}(w)+\frac{1}{m},$$
 which implies that $\limsup_{n\in {\mathbb N}}\tilde{\Gamma}(w_n)\leq \tilde\Gamma (w)$.
  

To prove the other inequality, let $\gamma:=\liminf_{n\in {\mathbb N}}\tilde{\Gamma}(w_n)$, and recall that, by assumption, $\delta_\Box(w_n,w)\rightarrow 0$ as $n\rightarrow \infty$.
Fix $0<\epsilon<1$, and let $n\in {\mathbb N}$ be chosen such that it satisfies
$$
 \delta_\Box(w_n,w)<\frac{\epsilon^2}{18^2}, \mbox{ and }
|\tilde{\Gamma}(w_n)-\gamma|<\frac{\epsilon}{3}.
$$
In addition, let $w'_n\in{\mathcal W}_0$ be such that $\delta_\Box(w'_n,w_n)=0$ and $|\Gamma(w'_n)-\tilde{\Gamma}(w_n)|<\epsilon/3$. By definition of the $\delta_\Box$-distance, there exists $\phi\in \Phi$ such that $\|w'_n-w^\phi\|_\Box<\frac{\epsilon^2}{18^2}$ and thus, by Lemma \ref{lem:Gamma-cts}, $|\Gamma (w'_n)-\Gamma (w^\phi)|\leq 6\sqrt{\|w'_n-w^\phi\|_\Box}<\epsilon/3$. Thus,
\begin{eqnarray*}
|\Gamma(w^\phi)-\gamma |&\leq&
|\Gamma(w^\phi)-\Gamma(w'_n)|+|\Gamma(w'_n)-\tilde{\Gamma}(w_n)|+|\tilde{\Gamma}(w_n)-\gamma|\\
&<&\epsilon.
\end{eqnarray*}
Therefore, for every $0<\epsilon<1$, $\tilde\Gamma(w)\leq \Gamma(w^\phi)\leq \liminf_{n\in {\mathbb N}}\tilde{\Gamma}(w_n)+\epsilon$.
Combining this with the lower bound, we get $$\limsup_{n\in {\mathbb N}}\tilde{\Gamma}(w_n)\leq \tilde\Gamma (w)\leq\liminf_{n\in {\mathbb N}}\tilde{\Gamma}(w_n),$$ which implies that $\lim_{n\rightarrow\infty}\tilde{\Gamma}(w_n)=\tilde{\Gamma}(w)$.

\end{proof}

\begin{corollary}\label{cor:contin_Gamma*}
Let $w\in \cal W_0$ be the limit of a convergent sequence $\{G_n\}_{n\in {\mathbb N}}$ of graphs with $|V(G_n)|\rightarrow \infty$. Then  
 $\{\Gamma^*(G_n)\}_{n\in {\mathbb N}}$ converges to $\tilde\Gamma (w)$ as $n\rightarrow \infty$.
\end{corollary}

\begin{proof}
For each $n\in{\mathbb N}$, let $w'_{G_n}$ be the step function representing $G_n$ with respect to an ordering $\prec'_n$ that is optimal for $\Gamma^*$. Thus, by Corollary \ref{cor:gamma-graph-function},
$$\liminf_{n\in{\mathbb N}}\Gamma^*(G_n)=\liminf_{n\in {\mathbb N}}\Gamma^*(G_n,\prec'_n)=\liminf_{n\in{\mathbb N}}\Gamma(w'_{G_n})\geq 
\liminf_{n\in{\mathbb N}}\tilde\Gamma(w'_{G_n}).$$
Clearly the sequence $\{w'_{G_n}\}$ converges to $w$ with respect to $\delta_\Box$- distance. 
Thus by Theorem \ref{thm:Gamma-cts},   
$$\liminf_{n\in{\mathbb N}}\Gamma^*(G_n)\geq \tilde\Gamma(w).$$
On the other hand, let $u\in{\mathcal W}_0$ be an element equivalent to $w$ such that $\tilde{\Gamma}(w)+\epsilon\geq \Gamma(u)$. 
Since the sequence  $\{G_n\}$ converges to $u$, there is a labelling of vertices of graphs $G_n$, corresponding to an ordering $\prec_n$, for which 
$\|w_{G_n}-w\|_\Box\rightarrow 0$. Thus by Lemma \ref{lem:Gamma-cts}, we have $\Gamma(w_{G_n})\rightarrow \Gamma(w)$. Therefore by another application of Corollary \ref{cor:gamma-graph-function} we have,
$$\epsilon+\tilde{\Gamma}(w)\geq \Gamma(u)=\lim_{n\rightarrow\infty}\Gamma(w_{G_n})=\lim_{n\rightarrow \infty}\Gamma^*(G_n,\prec_n)\geq \limsup_{n\in{\mathbb N}} \Gamma^*(G_n).$$
%
\end{proof}

In particular, if a convergent graph sequence $\{ G_n\}$  with limit $w$ has the property that $\{\Gamma^*(G_n)\}$ converges to zero, the above theorem states that $\tilde{\Gamma}(w)=0$. This implies that there exist functions $u$ with $\Gamma (u)$ arbitrarily small so that the graphs $\{ G_n\}$ have similar structure, in terms of homomorphism densities, as the random graph $G(n,u)$. We would like to conclude that the graphs $\{ G_n\}$ are consistent with having been formed by a random process with a linear embedding. However, it does not follow from our results that any function $u$ with $\Gamma (u)$ small is \lq\lq close" to a diagonally increasing function. We conjecture that, in fact, if $\Gamma (w)$ is small, then there exists a diagonally function $u$ which is close to $w$ in box distance. 

\begin{conjecture}
There exists a strictly increasing function $f$ which approaches  zero as $x\rightarrow 0$ such that: 

For every $w\in {\mathcal W}_0$, there exists $u\in {\mathcal W}_0$ with $\Gamma (u)=0$ and  $\| w-u\|_{\Box}\leq f(\Gamma (w))$. 
\end{conjecture}


\section*{Acknowledgements}

The authors wish to thank AARMS, MPrime and NSERC for supporting this work. Ghandehari and Janssen acknowledge the American Institute of Mathematics (AIM) for an invitation to attend the workshop on graph and hypergraph limits, August 2011.  Parts of this article was written when the second author was supported by  Fields Institute for Mathematical Research as a postdoctoral fellow affiliated to the ``Thematic Program on Abstract Harmonic Analysis, Banach and Operator Algebras''. 
We also thank the anonymous referee, whose constructive comments helped improve the paper.


\end{document}